\documentclass{amsart}
\usepackage{latexsym}
\usepackage{amsmath,amssymb,amsfonts,amsthm,amscd}
\usepackage{tikz-cd}
\usepackage{geometry}
\usepackage{graphicx}
\usepackage{lscape}
\usepackage{enumerate}
\usepackage{mathrsfs}
\usepackage{comment}
\usepackage{arydshln}
\usepackage{stmaryrd}
\usepackage{mathtools}
\usepackage{braket}
\usepackage{color}
\usepackage[luatex,pdfencoding=auto]{hyperref}
\usepackage{aliascnt}
\usepackage{cleveref}
\crefname{section}{}{}
\creflabelformat{section}{\S#2#1#3}
\crefname{part}{Part}{Parts}
\crefname{equation}{}{}
\crefname{align}{}{}
\crefname{enumi}{}{}
\creflabelformat{enumi}{(#2#1#3)}
\newtheorem{thm}{Theorem}[section]
\crefname{thm}{Theorem}{Theorems}
\newaliascnt{cor}{thm}

\aliascntresetthe{cor}
\crefname{cor}{Corollary}{Corollaries}
\newaliascnt{prop}{thm}

\aliascntresetthe{prop}
\crefname{prop}{Proposition}{Propositions}
\newaliascnt{lem}{thm}

\aliascntresetthe{lem}
\crefname{lem}{Lemma}{Lemmas}
\newaliascnt{clm}{thm}

\aliascntresetthe{clm}
\crefname{clm}{Claim}{Claims}

\theoremstyle{definition}
\newaliascnt{defi}{thm}

\aliascntresetthe{defi}
\crefname{defi}{Definition}{Definitions}
\newaliascnt{ex}{thm}

\aliascntresetthe{ex}
\crefname{ex}{Example}{Examples}

\theoremstyle{remark}
\newaliascnt{rem}{thm}

\aliascntresetthe{rem}
\crefname{rem}{Remark}{Remarks}

\newcommand{\bC}{\mathbb{C}}

\newcommand{\bE}{\mathbb{E}}

\newcommand{\bQ}{\mathbb{Q}}

\newcommand{\bZ}{\mathbb{Z}}

\newcommand{\ol}{\overline}

\newcommand{\id}{\mathrm{id}}

\newcommand{\sing}{\mathrm{sing}}

\DeclareMathOperator{\Ima}{Im}

\DeclareMathOperator{\CH}{CH}

\title{Asymptotics of local height pairing without Hodge conjecture}
\author{Yuta Nakayama}
\date{\today}
\begin{document}

\subjclass[2020]{Primary: 14G40; Secondary: 14C25, 14C30}
\begin{abstract}
We remove the assumption of the Hodge conjecture from the asymptotic formula for the Archimedean local height established in our previous work.
We also remove other minor assumptions.
We achieve this by exploiting the cohomological nature of Be\u ilinson's local intersection index rather than arguing motivically.
\end{abstract}
\maketitle
\setcounter{tocdepth}{1}
\tableofcontents

\section{Introduction}

Suppose that \(X\to S\) is a family of smooth projective varieties over a Riemann surface that degenerates at \(s\in S(\bC)\).
Let \(W, Z\) be cycles whose codimensions sum up to \(\dim X\), whose supports do not meet in the generic fiber, and whose cycle classes in a general fiber vanish.
In \cite{NakAsym}, we proved, under additional assumptions, an asymptotic formula of the form
\begin{equation} \label{Main}
    h(t) - \langle W, Z\rangle_{\mathrm{a}, s}\log |t| = O(1),
\end{equation}
where \(t\) is a local coordinate around \(s\), the quantity \(h(t)\) is the Archimedean part of Be\u ilinson--Bloch height pairing, and \(\langle W, Z\rangle_{\mathrm{a}, s}\) is Be\u ilinson's local intersection index.
We remove unnecessary assumptions in this paper with the help of AI.
It identified a similar \(p\)-adic theorem by Scholl \cite[Theorem 5.3]{SchHei}.

Compared to our previous work, the argument here is simpler because it relies on cohomological rather than motivic arguments.
Previously, we aimed at a conceptual reformulation of \cref{Main} as a comparison of line bundles.
One of them is the line bundle by Bloch \cite{BloBiext} and Seibold \cite{SeiBier}, defined via higher Chow groups.
The other is that by Hain \cite{HainBiext} and Brosnan--Pearlstein \cite{BPJump}, based on Hodge theory and blended (or mixed) extensions, known as <<extensions panach\'ees>> in \cite[Expos\'e IX, \S 9.3]{SGA7}.
The comparison over the smooth locus uses the cycle class map from higher Chow groups to Deligne cohomology and by the isomorphism between Deligne cohomology and absolute Hodge cohomology.
On the singular locus, we needed to use mixed Hodge modules.

In our previous work, we assumed that all the special fibers of \(X\to S\) were reduced simple normal crossing divisors.
We also imposed certain proper-intersection hypotheses.
These would be manageable by blowups.
However, the most substantial assumption in the theorem was the Hodge conjecture.
This was used to modify \(W\) to a rational cycle \(\hat{W}\) whose cycle classes on all irreducible components of all fibers vanish.
We needed \(\hat{W}\) when we formulated Bloch's line bundle as \(\bE_{\hat{W}, Z}\), a fiber of Bloch's biextension.
On the other hand, the formulation as a comparison of line bundles becomes a little artificial without the conjecture.

We recall Hain's blended extension of variations of Hodge structures, obtained from two disjoint cycles of complementary dimensions, in \cref{sec:recall}.
We present the new argument in \cref{sec:main}.

\subsection*{Acknowledgments}
The author thanks his advisor S.-W. Zhang for suggesting that the author should remove the Hodge conjecture using AI.

\subsection*{AI usage}
The argument was outlined by ChatGPT 6 Astra.
The author worked out the remaining minor details and wrote the manuscript.
ChatGPT 6 Astra proofread the manuscript.

\section{Notations}
If \(f\colon X\to Y\) is a morphism of schemes over \(\bC\), we also write \(f\colon X(\bC)\to Y(\bC)\).
For a smooth quasi-projective scheme \(T\) over \(\bC\), we also have the cycle class maps
\[
    *_\sing \colon \CH^p(T)\to H^{2p}(T(\bC), \bZ(p))
\]
for Betti cohomology.
We use the same notation for the corresponding cycle class maps with support.
For an algebraic cycle \(W\), let \(|W|\) be its support.

In general, our sign conventions follow \cite{HartsRes}.

\section{Hain's blended extension} \label{sec:recall}

We recall the admissible ``biextension'' of variations of Hodge structures by Hain \cite{HainBiext} and Brosnan--Pearlstein \cite{BPJump}.
We follow \cite[\S 9]{NakAsym}.
Let \(m, n\) be nonnegative integers such that \(m + n = d + 1\).
Suppose that \(\pi\colon X\to S\) is a smooth projective morphism of relative dimension \(d\) of smooth \(\bC\)-schemes.
Assume that \(W\) (resp.~\(Z\)) is a cycle of codimension \(m\) (resp.~\(n\)) that intersects properly with the fibers of \(\pi\).
Also assume \(|W|\cap |Z| = \emptyset\).

First, set \(i_Z\colon |Z|\to X\) and \(j_Z\colon X\backslash |Z|\to X\) to be the inclusions.
We have the distinguished triangle
\[
    R\pi_*i_{Z, *}Ri_Z^!\bZ(n)\to R\pi_*\bZ(n)\to R\pi_{Z, *}\bZ(n)\xrightarrow{+1},
\]
where we let \(\pi_Z\coloneqq \pi\circ j_Z\).
This induces the first row of the commutative diagram
\begin{equation} \label{E_Z}
\begin{tikzcd}[cramped, column sep=small]
R^{2n-1}\pi_* i_{Z,*} Ri_Z^!\bZ(n) \arrow[r] &
R^{2n-1}\pi_* \bZ(n) \arrow[r] &
R^{2n-1}\pi_{Z,*}\bZ(n) \arrow[r] &
R^{2n}\pi_* i_{Z,*} Ri_Z^!\bZ(n) \arrow[r] &
R^{2n}\pi_* \bZ(n) \\
0 \arrow[r] \arrow[u, equal]&
R^{2n-1}\pi_* \bZ(n) \arrow[r] \arrow[u, equal] \arrow[d] &
E'_Z \arrow[r] \arrow[u] \arrow[d] &
\bZ_{S(\bC)} \arrow[r] \arrow[u, "{-[Z]}_{\mathrm{sing}}"] &
0 \arrow[u] \\
0 \arrow[r] &
\displaystyle\frac{R^{2n-1}\pi_* \bZ(n)}{\mathrm{tors}} \arrow[r] &
E_Z \arrow[r] &
\bZ_{S(\bC)} \arrow[r] \arrow[u, equal] &
0
\end{tikzcd}
\end{equation}
made of exact sequences, where we define the second row by the pullback by
\[
    \bZ_{S(\bC)}\xrightarrow{-[Z]_{\sing}} H^{2n}_{|Z|(\bC)}(X(\bC), \bZ(n))_{S(\bC)}\to R^{2n}(\pi\circ i_Z)_*Ri_Z^! \bZ(n).
\]
The cup product
\[
    R^{2n-1}\pi_{Z, *}\bZ (n)\times R^{2m-1}\pi_{Z, !}\bZ (m)\to R^{2d}\pi_{Z, !} \bZ(d + 1)\to \bZ(1)_{S(\bC)}
\]
induces
\begin{align} \label{2dual}
    R^{2m-1}\pi_{Z, !}\bZ (m)\to (R^{2n-1}\pi_{Z, *} \bZ(n))^\vee \to E_Z^\vee,
\end{align}
where \(*^\vee\) is the duality with respect to \(\bZ(1)_{S(\bC)}\).

Second, set \(i_W\colon |W|\to X\), \(j_W\colon X\backslash |W|\to X\) and \(j\colon X\backslash (|W|\cup |Z|)\to X\backslash |Z|\) to be the inclusions.
Consider the distinguished triangle
\[
    R\pi_*i_{W, *}Ri_W^!\bZ(m)\to R\pi_{Z, !}\bZ(m)\to R\pi_{Z, !}Rj_*\bZ(m)\xrightarrow{+1}.
\]
This gives rise to the exact sequence appearing in the top row of the commutative diagram
\[
\begin{tikzcd}[cramped]
    0 \arrow[r] & R^{2m-1}\pi_{Z, !}\bZ(m) \arrow[r] \arrow[d, equal] & R^{2m-1}\pi_{Z, !}Rj_*\bZ(m) \arrow[r] & R^{2m}\pi_* i_{W, *}Ri_W^!\bZ(m) \arrow[r] & R^{2m}\pi_{Z, !} \bZ(m) \\
    0 \arrow[r] & R^{2m-1}\pi_{Z, !}\bZ(m) \arrow[r]                  & B_W(Z) \arrow[u] \arrow[r]                  & \bZ_{S(\bC)} \arrow[r] \arrow[u, "{-[W]_{\sing}}"] & 0. \arrow[u]
\end{tikzcd}
\]
The pullback as in \cref{E_Z} defines the second row, and the rightmost square commutes for dimension reasons.

We now form the pushout
\[
\begin{CD}    
    0 @>>> R^{2m-1}\pi_{Z, !}\bZ(m)@>>> B_W(Z) @>>> \bZ_{S(\bC)} @>>> 0 \\
    @.    @V{\cref{2dual}}VV           @VVV        @|              @.\\
    0 @>>> E_Z^\vee                @>>> B_{W, Z} @>>> \bZ_{S(\bC)} @>>> 0.
\end{CD}
\]
Poincar\'e duality gives the short exact sequence
\[
    0\to \bZ(1)_{S(\bC)}\to E_Z^\vee \to R^{2m-1}\pi_*\bZ (m)/\mathrm{tors} = H \to 0.
\]
We equip \(B_{W, Z}\) with the following structure of a blended extension:
\begin{equation} \label{signZ}    
    0\to \bZ(1)_{S(\bC)}\xrightarrow{-1} E_Z^\vee \to B_{W, Z}.
\end{equation}
We have \(B_{W, Z}/\bZ(1) = E_W\), where \(E_W\) is defined analogously to \(E_Z\).

\section{Main theorem} \label{sec:main}

Assume that \(S\) is a smooth irreducible curve over \(\bC\).
Take \(\pi\colon X\to S\), a flat projective morphism of schemes of relative dimension \(d\) such that \(X\) is smooth over \(\bC\).
Let \(W\) and \(Z\) be cycles in \(X\) of codimensions \(m\) and \(n\), respectively, with \(m+n=d+1\).
Assume that they are homologically trivial on general fibers and that they do not meet in the generic fiber.

Take a closed point \(a\colon s\to S\).
Let \(b\colon S\backslash \{s\}\to S\) be the open immersion.
We identify sheaves on \(\{s\}\) with abelian groups.
Similarly to \cite[Lemma 2.2.2.c)]{BeiHeiCyc}, we have
\[
    [W]_{\sing} = 0\in a^*R^{2m}b_* R\pi_*\bQ(m).
\]
This enables us to take
\[
    \alpha_1\in a^*R^{2m-1}b_* R\pi_{W, *}\bQ(m)
\]
that maps to \([W]_{\sing}\in a^*R^{2m}b_* R\pi_*i_{W, *}Ri_W^!\bQ(m)\) under the connecting homomorphism, where \(\pi_W \coloneqq \pi\circ j_W\).
Our \(\langle W, Z\rangle_{\mathrm{a}, s}\) is the image of \(\alpha_1\cup [Z]_{\sing}\) under the following map \cite[Lemma-Definition 2.1.1]{BeiHeiCyc} \cite[Proposition B.4]{LLMot}:
\begin{equation} \label{Around2.1.1}
    a^*R^{2d + 1}b_* R\pi_{W, *}j_W^*i_{Z, *}Ri_Z^!\bQ(d + 1)\to a^*R^{2d + 1}b_* R\pi_*\bQ(d + 1) \to a^*R^{1}b_* \bQ(1) = \bQ.
\end{equation}

\begin{thm}
Suppose that \(S\) is a smooth irreducible curve over \(\bC\).
Also suppose that \(\pi\colon X\to S\) is a flat projective morphism of schemes of relative dimension \(d\) such that \(X\) is smooth over \(\bC\).
Let \(W\) and \(Z\) be cycles in \(X\) of codimensions \(m\) and \(n\), respectively, with \(m+n=d+1\).
Assume that they are homologically trivial on general fibers and that they do not meet in the generic fiber.
Take a closed point \(s\in S\).
Let
\[
    p\colon \{z\in \bC\mid |z|<1\}\to S(\bC)
\]
be a local coordinate in a sufficiently small region around \(p(0) = s\).
For \(t\neq 0\) in the source of \(p\), let \(h(t)\) be the height pairing of \(W_{p(t)}, Z_{p(t)}\) in the sense of \cite[p.~860]{HainBiext}.
Then
\[
    h(t) - \langle W, Z\rangle_{\mathrm{a}, s}\log |t|
\]
extends continuously to \(t = 0\), where \(\langle W, Z\rangle_{\mathrm{a}, s}\) is the local intersection index defined in \cite[\S 2]{BeiHeiCyc}.
\end{thm}

\begin{proof}
By \cite[(3.3.12)]{HainBiext}, the Archimedean height is obtained as the height of the blended extension \(B_{W,Z,\mathbb Q}\) in the sense of \cite{BPJump}.
Put \(V\coloneqq B_{W, Z, \bQ, p(t)}\) for some \(t\neq 0\), well-defined as long as the image of \(p\) is sufficiently small.
Set \(T_V\) to be the monodromy of \(V\).
More generally, we use \(T\) as a notation for monodromy around \(s\).
Let \(N\) be a positive integer such that the \(N\)-th power of the monodromy on \(R^{2m - 1}\pi_*\bQ(m)\) around \(s\) is unipotent.
As in \cite[Remark 20]{BPJump}, there exist \(v_0\in V\) and a subspace \(U\subset V\) such that
\begin{enumerate}
    \item \(V = \bQ v_0\oplus U\oplus \bQ(1)\),
    \item \((T_V^N - \id)(U) \subseteq U\),
    \item\label{MuDef} \(T_V^N(v_0) - v_0 \in \bQ(1)\), and
    \item \(v_0\mapsto 1\), \(U\to 0\) by the surjection \(V \to \bQ\).
\end{enumerate}
It suffices to show that \(T_V^N(v_0) - v_0 = -2\pi \sqrt{-1}N\langle W, Z\rangle_{\mathrm{a}, s}\).

Put
\begin{align*}
p_N&\colon \{z\in \bC\mid 0<|z|<1\}\to S(\bC), \quad z\mapsto p(z^N).
\end{align*}
Let \(\ol{\alpha_1}\) be the image of \(v_0\) under \(B_{W_{p(t)}, Z_{p(t)}, \bQ}\to E_{W_{p(t)}, \bQ}\) for some \(t\neq 0\) in the disk.
The condition \cref{MuDef} amounts to saying that \(\ol{\alpha_1}\) is invariant under \(T^N\), defining \(\ol{\alpha_1}\in \Gamma p_N^*E_\bQ\) as well as \(\ol{\alpha_1}\in \Gamma p_N^* R^{2m-1}\pi_{W, *}\bQ(m)\).
Let \(\alpha_1\) be a lift of this element to \(R^{2m-1}\Gamma  p_N^* R\pi_{W, *}\bQ(m)\).

We have \(R\Gamma p_N^* L = (\dots 0\to L_{p(t)}\xrightarrow{T^N - 1} L_{p(t)}\to 0\dots)\) for a local system \(L\) on the image of \(p\).
Therefore, \(T^N v_0 - v_0\) can be seen as the image of \(\ol{\alpha_1}\) under the connecting homomorphism \(\Gamma p_N^* E_\bQ\to R^1 \Gamma p_N^* \bQ(1) = \bQ\) associated with the exact sequence
\begin{equation} \label{B_WZ2E_W}
0 \to \bQ(1) \to B_{W, Z, \bQ} \to E_{W, \bQ} \to 0,
\end{equation}
where the sign in the standard convention but not in the proof of the snake lemma \cite[(1.3.3)]{ConBaseChange} is cancelled by the sign in the definition of \(E_W\) in \cref{E_Z}.

It suffices to show that this image of \(\ol{\alpha_1}\) in \(\bQ\) coincides with the description of \(\langle W, Z\rangle_{\mathrm{a}, s}\) around \cref{Around2.1.1} up to a factor of \(N\).

As in \cite[(32)]{NakAsym}, the sequence \cref{B_WZ2E_W} is obtained by pulling back the second row of the pushout diagram
\[
\begin{tikzcd}[cramped, column sep=small]
        R^{2m-2} \pi_{W, *}\bQ(m)
                \arrow[r, "{\mathrm{res}}"]
                \arrow[d]
            & R^{2m-2}(\pi\circ i_Z)_*\bQ(m)
                \arrow[r]
                \arrow[d, "{[Z]_{\sing}}"]
            & R^{2m-1} \pi_{W, *}Rj'_!\bQ(m)
                \arrow[r]
                \arrow[d]
            & R^{2m-1} \pi_{W, *}\bQ(m)
                \arrow[r]
                \arrow[d, equal]
            & 0 \\
        0
                \arrow[r]
            & \bQ(1)
                \arrow[r]
            & *
                \arrow[r]
            & R^{2m-1} \pi_{W, *}\bQ(m)
                \arrow[r]
            & 0
\end{tikzcd}
\]
of exact sequences over the punctured image of \(p\), with the following explanations.
\begin{itemize}
\item In this diagram, \([Z]_{\sing}\) appears without a minus sign because the minus signs in \cref{E_Z} and \cref{signZ} cancel.
    The sign discrepancy denoted by \(+\) in \cite[(32)]{NakAsym} comes from the difference between \(B_{W, Z}\) and \(B_{W, Z}^\circ\) in \cite{NakAsym}.
\item We put \(j'\colon X\backslash (|W|\cup |Z|)\to X\backslash |W|\).
    We have the isomorphism \(j_{Z, !}Rj_*j^*j_Z^*\simeq Rj_{W, *}j'_!j'^*j_W^*\) by applying the snake lemma to
    \[
        \begin{CD}
            0 @>>> i_{W, *}i_W^! @>>> \id @>>> j_{W, *}j_W^* @. \\
            @.      @VVV            @|         @VVV         @. \\
            0 @>>> j_{Z, !}j_Z^* @>>> \id @>>> i_{Z, *}i_Z^* @>>> 0.
        \end{CD}
    \]
\end{itemize}
The connecting homomorphism in question is therefore induced by the lower-left vertical arrow in the diagram
\[
    \begin{tikzcd}
        R^{2m-1}\Gamma p_N^*R\pi_{W, *}\bQ(m) \arrow[d] \arrow[rd] \arrow[r, "{[Z]_{\sing}}"]& R^{2d + 1}\Gamma p_N^*R\pi_{W, *}j_W^*i_{Z, *}Ri_Z^!\bQ(d + 1) \arrow[rd, equal] & \\
        \Gamma p_N^*R^{2m-1}\pi_{W, *}\bQ(m) \arrow[d]                                     & R^{2m-1}\Gamma p_N^*R(\pi\circ i_Z)_*\bQ(m) \arrow[r, "{[Z]_{\sing}}"] \arrow[ld] & * \\
        R^1\Gamma p_N^* (R^{2m-2}(\pi\circ i_Z)_*\bQ(m)/\Ima \mathrm{res}) \arrow[r, "{[Z]_{\sing}}"] & R^1\Gamma p_N^*R^{2d}\pi_{W, *}j_W^*i_{Z, *}Ri_Z^!\bQ(d + 1), \arrow[ru] &
    \end{tikzcd}
\]
with the following details.
\begin{itemize}
    \item The diagram commutes except for the triangle on the left, which anticommutes because of the sign convention in \cite[(TR2)]{HartsRes}.
    \item The down-left morphism is induced by the fact that \(R^{2m-1}(\pi\circ i_Z)_*\bQ(m) = 0\) so that \(R^1\Gamma p_N^* R^{2m-2}(\pi\circ i_Z)_*\bQ(m) = R^{2m-1}\Gamma p_N^*R(\pi\circ i_Z)_*\bQ(m)\).
\end{itemize}
Our \(\alpha_1\) belongs to the stalk of the sheaf in the upper-left corner of the diagram.
Its image under the vertical map is \(\ol{\alpha_1}\).
Starting from \(\alpha_1\), going to the right and then down gives us \(\langle W, Z\rangle\in \bQ = R^1\Gamma \bQ(1)\), up to multiplication by \(N\) due to \(p_N^*\).
Following the two downward arrows and then the rightward arrow gives us the desired rational number in \(\bQ = R^1\Gamma \bQ(1)\).
\end{proof}


\end{document}